\documentclass{amsart}
\usepackage{amsmath,amssymb,amsfonts,amsthm,enumerate,mathtools}
\usepackage{tikz}
\usetikzlibrary{calc}
\tikzset{vnode/.style= {draw,circle,fill=black!50, minimum width=4pt, inner sep=0pt}}

\usepackage{float}
\usepackage{graphicx}
\usepackage{subfigure}
\usepackage{xspace}
\usepackage[hidelinks]{hyperref}

\newtheorem{theorem}{Theorem} 
\newtheorem{lemma}[theorem]{Lemma}
\newtheorem{corollary}[theorem]{Corollary}
\newtheorem{proposition}[theorem]{Proposition}

\theoremstyle{definition}
\newtheorem{example}[theorem]{Example}
\newtheorem{openprob}{Problem}

\numberwithin{equation}{section}

\newcommand{\aut}[1]{\ensuremath{\operatorname{Aut}\,#1}\xspace}
\renewcommand{\ker}[1]{\ensuremath{\operatorname{ker}\,#1}\xspace}

\newcommand{\orcycle}[1]{\ensuremath{#1}\xspace}
\newcommand{\cycle}[1]{\ensuremath{\mathbf{#1}}\xspace}

\begin{document}
\title{On a representation of the automorphism group of a graph in a unimodular group}
\author[I. Est\'elyi]{Istv\'an Est\'elyi}
\address[I. Est\'elyi]{NTIS, University of West Bohemia, Technick\'{a} 8, 30100 Plze\v{n} 3, Czech republic}
\address[I. Est\'elyi]{Faculty of Information Technology, University of Pannonia, Egyetem u. 10., 8200 Veszprém, Hungary}
\email[I. Est\'elyi]{estelyi.istvan@mik.uni-pannon.hu}
\author[J. Karab\'a\v{s}]{J\'an Karab\'a\v{s}}
\address[J. Karab\'a\v{s}]{Faculty of Natural Sciences, Matej Bel University, Tajovsk\'eho~40, 97401 Bansk\'a Bystrica, Slovakia}
\address[J. Karab\'a\v{s}]{NTIS, University of West Bohemia, Technick\'{a} 8, 30100 Plze\v{n} 3, Czech republic}
\email[J. Karab\'a\v{s}]{karabas@savbb.sk}
\author[R. Nedela]{Roman Nedela}
\address[R. Nedela]{Faculty of Applied Sciences, University of West Bohemia, Technick\'{a} 8, 30100 Plze\v{n} 3, Czech republic}
\address[R. Nedela]{Mathematical Institute, Slovak Academy of Sciences, \v{D}umbierska~1, 97411~Bansk\' a Bystrica, Slovakia}
\email[R. Nedela]{nedela@ntis.zcu.cz}
\author[A. Mednykh]{Alexander Mednykh}
\address[A. Mednykh]{Sobolev Institute of Mathematics, Pr. Koptyuga 4, Novosibirsk, 630090, Russia}
\address[A. Mednykh]{Novosibirsk State University Pirogova 2, Novosibirsk, 630090, Russia}
\email[A. Mednykh]{smedn@mail.ru}

\begin{abstract}  
We investigate a representation of the automorphism group of a connected graph $X$ in the group of unimodular matrices $U_\beta$ of dimension $\beta$, where $\beta$ is the Betti number of graph $X$. We classify the graphs for which the automorphism group does not embed into $U_\beta$. It follows that if $X$ has no pendant vertices and $X$ is not a simple cycle, then the representation is faithful and $\aut{X}$ acts faithfully on $H_1(X,\mathbb{Z})$. The latter statement can be viewed as a discrete analogue of a classical Hurwitz’s theorem on Riemann surfaces of genera greater than one.
\end{abstract}

\keywords{
graph, automorphism, unimodular matrix}


\subjclass[2010]{05C25 (Primary),~20F65,~20C10,~20H25 (Secondary)}

\maketitle

\section{Introduction}
\noindent{}Given any topological space $X$ and a finite group $G$ of homeomorphisms of $X$, consider the associated action of $G$ on the first homology group $H_1(X, \mathbb{Z})$.
It is natural to ask when this representation is faithful. Hurwitz’s theorem says when $X$ is a Riemann surface of genus greater than one, the answer is yes,
 see Farkas and Kra~\cite[Theorem~V.3.1,~p.~270]{fkra92}.
 A. M. MacBeath in~\cite{macb73} proved that the action remains faithful on the homology group $H_1(X, \mathbb{Z}_p)$, $p > 2$, but fails for $H_1(X, \mathbb{Z}_2)$, see also~\cite[p.~276]{fkra92}. 
In this paper we consider the case where $X$ is a finite graph.
A motivation to investigate it  comes from
the applications of the matrix representation of the automorphism group of a graph in the lifting automorphism problem~\cite{malnivc2004elementary, malnic20009, siran2001coverings}.
Recently we have found an application of the result in investigation of structure of Jacobian of a graph, see \cite{baker2007riemann}
for definitions and further details.

Throughout  we assume that $X$ is a \emph{simple connected} graph. By a \emph{dart} of
$X$ we mean an edge endowed with one of the two possible orientations. Thus every edge $uv$ gives rise to two darts $(u,v)$ and $(v,u)$ distinguished just by the orientation. If $x$ is dart, then the oppositely oriented dart underlying the same edge will be denoted by $x^{-1}$. An \emph{unoriented simple cycle} is a connected subgraph of $X$, where every vertex has degree $2$. Any simple cycle has two orientations $C$, $C^{-1}$ which we call \emph{oriented (simple) cycles}. 

Let $T$ be a spanning tree of $X$. For every co-tree edge $e_i$
choose one of the two underlying darts, and denote it by $x_i$.  
Clearly, each $x_i$ determines a unique simple oriented cycle $\orcycle{C}_i=(x_i,y_1,y_2,\dots,y_k)$, where $y_1y_2\dots y_k$ is the unique path in $T$ joining the terminal vertex of $x_i$ to the initial vertex of $x_i$. The cycles $C_i$, $i=1,\ldots,\beta$ will be called \emph{fundamental oriented cycles}. They form a basis of \emph{first homology group} $H_1(X)$ which is isomorphic to free abelian group of rank $\beta$, where $\beta(X)$ is the Betti number.

An \emph{automorphism}
of a graph is a permutation of the vertices taking
adjacent vertices onto adjacent vertices. In this note we consider as a rule the induced action of graph automorphisms on the set of darts. Hence automorphisms
of a graph are considered as permutations of the set of darts.
 For every automorphism $f\in \aut{X}$ we can construct a matrix $M=M_T(f)$ with
entries in $\{1,0,-1\}$ as follows. Fix a linear order of the
co-tree edges $e_1,e_2,\dots,e_\beta$, where $\beta=\beta(X)=e(X)-v(X)+1$. The automorphism $f$ of the graph $X$ induces a linear transformation of $H_1(X)$ which we can treat as a vector space of dimension $\beta$ (with integer scalars). Then $M_T(f)$ is the usual matrix of a linear transformation with respect to the basis given by the fundamental oriented cycles $C_i$ determined by co-tree darts of $T$.  In particular, the entry $m_{i,j}$ of $M_T(f)$ takes value $1$ or $-1$
if and only if $x_i$ or $x_i^{-1}$ is a dart traversed by $f(C_j)$, respectively. Finally, choosing a different spanning tree $T'$  can be viewed as a change of basis of the vector space $H_1(X)$. Hence, $M_{T'}(f)$ is a conjugate of $M_T(f)$ by a matrix representing the coordinates for the new cycles in terms of the old cycles. In particular, the
issue of faithfulness is obviously independent of the choice of spanning tree.

%
%

Recall that a square matrix is \emph{unimodular} if its determinant is $\pm1$. 
The unimodular matrices of dimension $\beta$ (with standard multiplication) form  a group, here  denoted $U_\beta$. By convention, if $\beta=0$, the group $U_\beta$ is trivial.
The following result is well-known.

\begin{theorem}[\cite{siran2001coverings}] The assignment $\Theta_T: f\mapsto M_T(f)$ defines a homomorphism of $\aut{X}$ into the group of unimodular matrices.
\end{theorem}

Observe that if $X$ is a tree with a non-trivial automorphism group, 
or if $X$ is a simple cycle, then $\Theta=\Theta_T$ is not injective for any spanning tree $T$ of $X$.
In what  follows, 
we consider the following problem:

\smallskip
\noindent{\bf Problem.} For which connected graphs $X$ does the homomorphism $\Theta_T$ determine an embedding of $\aut{X}$ into the group $U_\beta$? 
\smallskip

\section{Main result}
\noindent{}We will see that the answer to the posed problem depends on the structure of blocks of $X$. 
To this end, we recall a few related basic results on the $2$-connectivity of graphs, that can be found e.g. in monographs~\cite[Chapter~5]{bondy2011graph} and~\cite[Chapter~3]{diestel2005graph}. A graph with at least two vertices is \emph{$2$-connected}, if it cannot be disconnected by the removal of a single vertex, or in other words, if it has no \emph{cutvertices}. A \emph{bridge} in a connected graph is an edge whose removal disconnects the graph. A bridgeless connected graph is \emph{$2$-edge-conected}. Clearly, a $2$-connected graph is $2$-edge-connected.

By a \emph{block} of a graph $X$ we mean a maximal $2$-connected induced subgraph of $X$. A block with at least three vertices is called \emph{nontrivial}. Blocks of an arbitrary simple graph determine a decomposition of the set of edges. More precisely, the following properties hold true:

\begin{enumerate}[\hspace{1em}(P1)]
\item any two blocks of $X$ have at most one vertex in common,
\item every edge of $X$ belongs to exactly one block of $X$, 
\item each cycle of $X$ is contained in exactly one block of $X$.  
\end{enumerate}
The \emph{block tree} $\mathcal{B}=\mathcal{B}(X)$ of a connected graph $X$ is a bipartite graph with the vertex set $V(\mathcal{B})$ formed by the blocks of $X$ and by the cutvertices of $X$. A cutvertex $v\in V(\mathcal{B})$ is adjacent to a block $B\in V(\mathcal{B})$ if and only if $v\in B$. It can be easily seen that $\mathcal{B}$ is a tree~\cite[Proposition~3.1.2]{diestel2005graph}. All leafs of $\mathcal{B}$ are blocks. It is well-known~\cite[Chapter~3]{diestel2005graph} that the centre of $\mathcal{B}$ is a single vertex, 
that can either be a block or a cutvertex of $X$. In case it is a block, it is called the \emph{central block} of $X$. Since an automorphism $f\in\aut{X}$ maps blocks onto blocks and permutes the cutvertices, $f$ induces an automorphism $f^*$ of $\mathcal{B}(X)$. Clearly, $f^*$ fixes the central vertex of $\mathcal{B}$ and $f^*(v) = f(v)$ if $v$ is a cutvertex.

\begin{lemma}\label{lem:nova1}
Let $T$ be an arbitrary spanning tree of $X$, let $\orcycle{C}_i\neq \orcycle{C}_j$ be fundamental oriented cycles with respect to $T$ and let
$f\in\ker \Theta_T$. Then $\orcycle{C}_i$, $i=1,2,\ldots,\beta(X)$ are fixed by $f$ setwise. Moreover, if $\orcycle{C}_i\cap \orcycle{C}_j\neq\emptyset$, then $\orcycle{C}_i\cup \orcycle{C}_j$ is fixed pointwise.
\end{lemma}
\begin{proof}
Since $M_T(f)=\operatorname{id}$, $f(\orcycle{C}_i)$ traverses exactly one co-tree dart $x_i$ determining $\orcycle{C}_i$. Thus $f(\orcycle{C}_i)=\orcycle{C}_i$. Assume $\orcycle{C}_i\cap \orcycle{C}_j$ is a path $P$, possibly of length zero. By the previous statement, $P$ is fixed pointwise. Clearly, there exist a dart in $C_i-P$ incident to an endvertex of $P$. Such a dart is  fixed by $f$. An automorphism of a simple cycle $C$ fixing a dart fixes $C$ pointwise.  
\end{proof}

\begin{corollary}\label{cor:newcor1}
If $f\in\ker \Theta_T$, then $f(B)=B$ for each nontrivial block $B$ in $X$.
\end{corollary}

\begin{lemma}\label{lem:newlema2}
A maximal $2$-edge-connected subgraph of a connected graph $X$ that is not a simple cycle is fixed pointwise by $f\in\aut{X}$.
\end{lemma}

\begin{proof} 
Let $Z$ be a subgraph of $X$ satisfying the assumptions. Let $Y$ be a graph whose vertices are oriented fundamental cycles of $Z$. Two such cycles are adjacent if they have nontrivial intersection. Since $Z$ has no bridges, $Y$ is connected. By the assumptions $Y$ contains more than one vertex. Using induction based on Lemma~\ref{lem:nova1}, $f$ is the identity on $Z$.
\end{proof}

\begin{proposition}\label{prop:nova2} 
Let $X$ be a simple connected graph without vertices of degree one. Then either the homomorphism $\Theta_T\colon \aut{X}\to U_\beta$ is injective or $X$ is a simple cycle.
\end{proposition}

\begin{proof} 
If $X$ is $2$-edge-connected, the statement follows from Lemma~\ref{lem:newlema2}. 

Suppose now that $X$ has bridges. Consider the block tree $\mathcal{B} = \mathcal{B}(X)$. By assumptions, $\mathcal{B}(X)$ has at least $3$ vertices. Since there are no vertices of degree one, every leaf in $\mathcal{B}$ represents a nontrivial block. By Corollary~\ref{cor:newcor1}, every leaf of $\mathcal{B}(X)$ is fixed. Since any automorphism  of a tree fixing all the leafs is the identity, we have $f^*=\operatorname{id}$. Hence, $f$ fixes all bridges of $X$. By Lemma~\ref{lem:newlema2}, $f$ also fixes blocks of $X$ that are not cycles pointwise. If a block $B$ is a cycle and $X$ has at least two blocks, then the unique simple path $P$ in $\mathcal{B}$ from $B$ to any block $B'$ is fixed by $f$. By definition of $\mathcal{B}$, $P$ traverses a cutvertex $u\in B$. Since $f$ fixes $u\in B$, by Lemma~\ref{lem:nova1}, $f\big| _B=\operatorname{id}$. Therefore, $f=\operatorname{id}$ on $X$.
\end{proof}

Let $w\in V(X)$ be a cutvertex of $X$ such that there exist a bridge $B_1$ and a nontrivial block $B_2\ncong K_2$, both containing $w$. By a \emph{pendant tree} of $X$ rooted at $w$ we mean a subgraph of $X$ induced by $\{w\}\cup V(F_w)$, where $F_w$ is the maximal acyclic subgraph of $X-w$. A pendant tree will be called \emph{rigid}, if its automorphism group is trivial. A unicyclic graph will be called \emph{periodic} if it admits a non-trivial automorphism $f$ rotating the unique cycle.

\begin{example}\label{ex:uniper}
All periodic unicyclic graphs can be constructed as follows. Choose two positive integers $n,k$ such that $n>2$, $k\mid n$, and $k<n$. Further, choose a sequence of rooted trees $S_0,S_1,\ldots,S_{k-1}$ with the corresponding roots $w_0, w_1,\ldots,w_{k-1}$. Let $\cycle{C} = (v_0,v_1,\ldots,v_{n-1})$ be a simple cycle with vertices $v_i$, $i\in\{0,1,\ldots,n-1\}$. Form a graph $X$ by taking the cycle $\cycle{C}$ and attach to every vertex $v_j$ of $\cycle{C}$ a copy of the tree $S_i$, where $j\equiv i\pmod{k}$, by identifying  $v_j$ with the root $w_i$ of $S_i$, see e.g. Figure~\ref{fig:uniper}. 

It is obvious that $X$ admits an automorphism $\varrho$ of order $\frac{n}{k}>1$ rotating $\cycle{C}$ with period $k$. Hence, $\varrho$ is a nontrivial element in the kernel of $\Theta_T$, for any spanning tree $T$ of $X$. Therefore, $\Theta_T$ is not injective.
\begin{figure}[H]
\centering
\begin{tikzpicture}[thick,scale=0.8]
\node[vnode] (a) at (0,-0.5) {};
\node[vnode] (b) at (0,0.5) {};
\draw (a) -- (b);
\node[draw=none] at (0,-2) {$S_0$};
\end{tikzpicture}
\hspace{2cm}
\begin{tikzpicture}[thick,scale=0.8]
\node[vnode] (c) at (0,-1) {};
\node[vnode] (d) at (0,0) {};
\node[vnode] (e) at (0,1) {};
\draw[vnode] (c) -- (d) -- (e);
\node[draw=none] at (0,-2) {$S_1$};
\end{tikzpicture}
\hspace{2cm}
\begin{tikzpicture}[thick,scale=0.8]
\node[vnode] (a1) at (-0.5,-0.5) {};
\node[vnode] (a2) at (-0.5,0.5) {};
\node[vnode] (a3) at (0.5,0.5) {};
\node[vnode] (a4) at (0.5,-0.5) {};
\node[vnode] (b1) at (-1,-1) {};
\node[vnode] (b2) at (1,1) {};
\node[vnode] (c1) at (-1,1) {};
\node[vnode] (c2) at (-1.5, 1.5) {}; 
\node[vnode] (c3) at (1,-1) {};
\node[vnode] (c4) at (1.5,-1.5) {};
\node[draw=none] at (0,-2) {$X$};

\draw (a1) -- (a2) -- (a3) -- (a4) -- (a1);
\draw (b1) -- (a1);
\draw (b2) -- (a3);
\draw (c2) -- (c1) -- (a2);
\draw (c4) -- (c3) -- (a4);
\end{tikzpicture}
\caption{An example of a periodic unicyclic graph $X$ with $n=4$ and $k=2$.}\label{fig:uniper}
\end{figure}
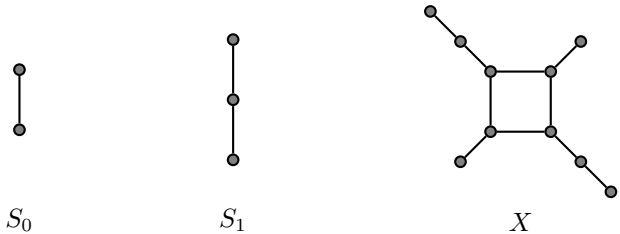

\end{example}

\noindent{}Now we are ready to prove the main result.

\begin{theorem}\label{thm:class}
Let $X$ be a connected simple graph and let $T$ be a spanning tree of $X$.
Then $\Theta_T$ is not injective if and only if at least one of the following statements holds.
\begin{enumerate}[(i)]
\item $X$ is a tree and $\aut{X}\neq 1$,
\item $X$ contains a pendant tree $S$ such that $\aut{S}\neq 1$, 
\item $X$ is periodic unicyclic.
\end{enumerate}
\end{theorem}
\begin{proof}
The direction $(\Leftarrow)$ is straightforward. If $X$ is a tree, then $\beta(X)=0$ and $U_\beta$ is trivial, hence $\Theta_T$ is not injective whenever $\aut{X}\neq 1$. If (ii) holds, then there exists an automorphism $\tau\in\ker{\Theta_T}$ fixing the complement of the pendant tree $S$ pointwise and acting nontrivially on $S$. The case (iii) is treated in Example~\ref{ex:uniper}.

$(\Rightarrow)$ In order to derive a contradiction
we assume that there exists a nontrivial automorphism $f\in\ker \Theta_T$ and none of (i), (ii), or (iii) holds true. In particular, we assume that $X$ is not a tree.

If $X$ contains some pendant trees $S_0$, $S_1$,\ldots, $S_{k-1}$, we reduce $X$ to a graph $X'$ by removing the trees $S_0$,\ldots,$S_{k-1}$. Since each of $S_i$ is rigid, we have
$\ker \Theta_T\big|_{X'}\geq \ker\Theta_T$. By Proposition~\ref{prop:nova2} $\ker \Theta_T\big|_{X'}$ is trivial.
\end{proof}

\section{Concluding remarks}
\noindent{}The aim of this section is to discuss several problems related to the main result. The first problem reads as follows. 
\begin{openprob}
What is the smallest dimension $d$ such that $\aut{X}$ is faithfully represented by $d\times d$ unimodular matrices?
\end{openprob}

Theorem~\ref{thm:class} characterises graphs $X$ for which $d\leq\beta(X)$. On the other hand, the automorphism group of a complete graph $K_n$, $n>3$, $\aut{K_n}\cong\mathrm{S}_n$, is faithfully represented by permutation matrices of type $n\times n$, while $\beta(K_n)=(n-1)(n-2)/2$. The other extreme is a cycle $C_p$, $p$ is prime, which automorphism group $\aut{C_p}\cong \mathrm{D}_{p}$, with faithful representation by $p\times p$ matrices while \mbox{$\beta(C_p)=1$}.

Another motivation comes from representation of groups. Given a group $G$ one can use Frucht's theorem ~\cite{Frucht1939} and ask for the graph $X$ such that $\aut{X}\cong G$. One can ask for `the smallest unimodular representation' of the group $G$.
\begin{openprob}
Given a finite group $G$ find the smallest graph $X$ such that $G\leq\aut{X}$ is faithfully represented by unimodular matrices.
\end{openprob} 
One can take $X$ to be a Cayley graph for $G$ and then ask the following.
\begin{openprob}
What is the smallest Betti number $\beta(X)$ through all Cayley graphs $X$ for $G$ such that the representation of $\aut{X}=G$ is faithful?
\end{openprob}
The question is related as well to the problem how to find the smallest generating set of $G$, giving the smallest valence of $X$.

A short analysis of our argumentation shows that the proof of faithfulness of the unimodular representation of $\aut{X}$ does not impose any requirements on the field over which the matrices are considered, except that $1\neq -1$. It follows that the representation remains faithful if the matrices $M_T(f)$ are considered over fields of characteristic $p>2$.     

\section*{Acknowledgements}
\noindent{}The authors express thanks to the anonymous referee for his/her useful comments which helped a lot to improve the presentation.

The first three authors were supported by the grant GACR 20-15576S. The first author acknowledges the financial support of Széchenyi 2020 under the EFOP-3.6.1-16-2016-00015 grant. The second and third author were supported by the grant No.~APVV-19-0308 of Slovak Research and Development Agency. The fourth author was supported by Mathematical Center in Akademgorodok under agreement No.~075-15-2019-1613 with the Ministry of Science and Higher Education of the Russian Federation.

\bibliographystyle{abbrv}
\bibliography{jacrep_literature}

\end{document}